\documentclass[11pt]{article}
\usepackage{amsmath,amsthm,amsfonts,amscd,bezier}
\usepackage{amssymb}
\usepackage{enumerate}
\usepackage[T1]{fontenc}
\usepackage[usenames]{color}

\newtheorem{lemma}{Lemma}[section]

\newtheorem{theorem}[lemma]{Theorem}
\newtheorem{remark}[lemma]{Remark}
\newtheorem{proposition}[lemma]{Proposition}
\newtheorem{corollary}[lemma]{Corollary}
\newtheorem{definition}[lemma]{Definition}
\newtheorem{example}[lemma]{Example}
\newcommand{\Dem}{\noindent{\sc Proof:\ \ }}
\newcommand{\cqd}{{\hfill $\rule{2mm}{2mm}$}\vspace{1cm}}

\title{On the Saito basis for plane curves}

\author{Emilio de Carvalho, Percy Fern\'{a}ndez-S\'{a}nchez  \\ and \\  Marcelo Escudeiro Hernandes\thanks{The third author was partially supported by CNPq.}}

\date{ \ }

\begin{document}
\maketitle


\begin{abstract} We present some results concerning the Saito module and the torsion submodule of an analytic plane curve, and we provide a method for computing them. Using this algorithm, we compute analytic invariants for plane curves with multiplicity less than or equal to three.\end{abstract}

\maketitle \markboth{E. Carvalho, P. Fern\'{a}ndez-S\'{a}nchez and M. E.
Hernandes}{On the Saito module and the Torsion submodule of analytic plane curves}

 Mathematics Subject Classification: Primary 14H20; Secondary 32S10.
 
  keywords: Saito module, Torsion submodule, plane curves, analytic invariants. 
  
\section{Introduction}

Let $\mathcal{C}_f$ be the germ of a reduced analytic plane curve defined by $f\in\mathbb{C}\{x,y\}$, that is
$$\mathcal{C}_f=\{(\alpha_1,\alpha_2)\in U;\ f(\alpha_1,\alpha_2)=0\ \mbox{for some neighborhood}\  U\ \mbox{at}\ (0,0)\in\mathbb{C}^2\}.$$ 

The topological type of $\mathcal{C}_f$ is completely characterized by its values semigroup $\Gamma(f)$, as shown in \cite{Delgado}, \cite{Waldi} and \cite{Z-top}.

In \cite{HR}, a solution to the problem of analytic classification of plane curves is presented by fixing a topological class and the set $\Lambda(f)$ of orders (along the curve) of the K\"ahler differential module $\Omega_f$ of the local ring $\mathcal{O}_f:=\frac{\mathbb{C}\{x,y\}}{\langle f\rangle}$.

The set $\Lambda(f)$ associated to a plane curve $\mathcal{C}_f$ is an analytical invariant and it is related to another analytic invariant: the Tjurina number $\tau(f):=\dim_{\mathbb{C}}\frac{\mathbb{C}\{x,y\}}{\langle f,f_x,f_y\rangle}$ (see \cite{Almiron} and \cite{comprimento} for instance).

On the other hand, Saito (see \cite{saito}) introduced the $\mathbb{C}\{x,y\}$-module of logarithmic 1-form $\Omega^1(\log\ \mathcal{C}_f)$ and the module of logarithmic vector fields $Der(log\ \mathcal{C}_f)$. Saito showed that $\Omega^1(\log\ \mathcal{C}_f)$ and $Der(log\ \mathcal{C}_f)$ are reflexive and free modules of rank 2 (see \cite[Corollary 1.7]{saito}). 

Considering $S(f):=f\cdot \Omega^1(\log\ \mathcal{C}_f)$ and using (1.1) pag. 266 in \cite{saito} we get
$$S(f):=\{\omega\in\Omega^1;\ \exists g, p\in\mathbb{C}\{x,y\}, \eta\in\Omega^1,\ gcd(g,f)=1\ \mbox{such that}\ g\omega=pdf+f\eta\},$$
where $\Omega^1:=\mathbb{C}\{x,y\}dx+\mathbb{C}\{x,y\}dy$ is the module of holomorphic 1-form at $(0,0)\in\mathbb{C}^2$. We call $S(f)$ the {\it Saito module} of $\mathcal{C}_f$.

Genzmer (see \cite{yohann}, \cite{yohann2} and \cite{yohann-marcelo}) has used the Saito module of $\mathcal{C}_f$ to obtain analytic data as the Tjurina number, the moduli number, etc. 

Let $\mathcal{K}_f$ denote the total ring of fractions of $\mathcal{O}_f$. Saito introduced the $\mathcal{O}$-module $\mathcal{R}(f)$ of residues of $S(f)$ given by
$$\mathcal{R}(f):=\left \{ \overline{\frac{p}{g}}\in \mathcal{K}_f;\ g\omega=pdf+f\eta\ \mbox{for}\ \omega\in S(f)\right \},$$
where the bar denotes the class of an element of $\mathbb{C}\{x,y\}$ in $\mathcal{O}_f$. The set $\Delta(f)$ of orders (along the curve) of the elements in $\mathcal{R}_f$ is an analytic invariant of $\mathcal{C}_f$ which determines and it is determined by $\Lambda(f)$ (see \cite{Pol}), and hence is also related to the Tjurina number of $\mathcal{C}_f$ (see \cite{Almiron}, \cite{transaction} and \cite{comprimento} for instance).

Since the Saito module $S(f)$ admits a basis $\{\omega_1, \omega_2\}\in\Omega^1$ (called a Saito basis), it is natural to explore properties of theses generators. As we mentioned, such generators provide analytical information about the curve. Moreover, $\omega_1$ and $\omega_2$ define saturated holomorphic foliations that leave $\mathcal{C}_f$ invariant, and their residues yield the GSV-index of these foliations. Presenting these and other properties of a Saito basis is one of the goals of Section 2.

There are some algorithms in the literature for computing a Saito basis (see \cite{tajima1}, \cite{tajima2} and \cite{walcher} for instance). In Section 3, we present an algorithm based on a Standard Basis for the K\"ahler differential module of the curve, which in turn can be computed using the results of \cite{carvalho}. Our algorithm may be viewed as a variation of the method presented in \cite{felipe-nuria-david} for the case of irreducible curves whose semigroup is generated by two elements.

In the Section 4, we apply our results to compute some analytic invariants ($\tau(f), \Lambda(f), S(f)$, etc.) for plane curves with multiplicity up to 3. These data may be of interest from the perspective of Commutative Algebra, Singularity Theory, and Foliation Theory. 

\section{The Saito module}

Let $f\in\mathbb{C}\{x,y\}$ be a reduced, non-unit holomorphic power series at $(0,0)\in\mathbb{C}^2$, and let $\mathcal{C}_f$ be the germ of the plane curve defined by $f$. The multiplicity of $f$ (or $\mathcal{C}_f$) is the positive integer $m(f)$ such that $f\in\langle x,y\rangle^{m(f)}\setminus\langle x,y\rangle^{m(f)+1}$. If $m(f)=1$, we say that $f$ (or $\mathcal{C}_f$) is regular.

In what follows, we denote by $\Omega^1=\mathbb{C}\{x,y\}dx+\mathbb{C}\{x,y\}dy$ the  
module of germs of holomorphic 1-form at the origin of $\mathbb{C}^2$ and we consider its submodule
$$F(f):=\{pdf+f\eta;\ p\in\mathbb{C}\{x,y\}\ \mbox{and}\ \eta\in\Omega^1\},$$
where $df=f_xdx+f_ydy$ is the differential of $f$.

Following the ideas of Saito in \cite{saito}, we define:

\begin{definition}\label{saito-module}
The {\bf Saito module} of $f$ (or $\mathcal{C}_f$) is
$$S(f)=\{\omega\in\Omega^1;\ g\omega\in F(f)\ \mbox{for some}\ g\in\mathbb{C}\{x,y\}\ \mbox{with}\ gcd(g,f)=1\}.$$
\end{definition}

In this section, we present some results concerning the Saito module of a plane curve $\mathcal{C}_f$, some of which may be known, but are included here with detailed proofs.

\begin{remark}\label{remark-equi-saito}
According to \cite[(1.1) page 266]{saito}, the condition
$g\omega\in F(f)$ for some $g\in\mathbb{C}\{x,y\}$ coprime with $f$, is equivalent to $\omega\wedge df=hfdx\wedge dy$ for some $h\in\mathbb{C}\{x,y\}$. We call the element $h$ the {\it cofactor of $\omega$}.
\end{remark}

\begin{lemma}\label{aux}
Let $\omega_1=A_1dx+B_1dy, \omega_2=A_2dx+B_2dy\in S(f)$ such that $\omega_i\wedge df=h_ifdx\wedge dy$. Then $\omega_1\wedge \omega_2=hfdx\wedge dy$ where $h=\frac{B_2h_1-B_1h_2}{f_y}=\frac{A_2h_1-A_1h_2}{f_x}\in\mathbb{C}\{x,y\}$.
\end{lemma}
\begin{proof} Since $\omega_i\wedge df=h_ifdx\wedge dy$, we have $A_if_y-B_if_x=h_if$. Thus,
$$(B_2h_1-B_1h_2)f=(A_1B_2-A_2B_1)f_y\ \ \mbox{and}\ \ (A_2h_1-A_1h_2)f=(A_1B_2-A_2B_1)f_x.$$
Hence, 
$$\frac{\omega_1\wedge\omega_2}{dx\wedge dy}=A_1B_2-A_2B_1=\frac{(B_2h_1-B_1h_2)f}{f_y}=\frac{(A_2h_1-A_1h_2)f}{f_x}.$$

Since $gcd(f,f_x)=gcd(f,f_y)=1$, it follows that $\omega_1\wedge\omega_2=hfdx\wedge dy$ with $h=\frac{B_2h_1-B_1h_2}{f_y}=\frac{A_2h_1-A_1h_2}{f_x}\in\mathbb{C}\{x,y\}$.
\end{proof}

The Saito module $S(f)$ is a free module of rank 2 (see Corollary 1.7 and Theorem page 270 in \cite{saito}). In addition, we have: 
\vspace{0.2cm}

\noindent {\bf (Saito's criterion)}\hspace{1cm} A pair $\{\omega_1,\omega_2\}$ is a $\mathbb{C}\{x,y\}$-basis for $S(f)$ if and only if \\

\vspace{-0.5cm}\hspace{3.8cm} $\omega_1\wedge\omega_2=ufdx\wedge dy\ \mbox{for some unit}\ u\in\mathbb{C}\{x,y\}.$
\vspace{0.2cm}

If $\{\omega_1,\omega_2\}$ satisfies the Saito's criterion then we call $\{\omega_1,\omega_2\}$ a {\bf Saito basis} for $S(f)$.

\begin{example}\label{qh}
Let $f$ be a quasi-homogeneous polynomial with weights $n, m$ and degree $\alpha$. By Euler's identity, we have $\alpha f=nxf_x+myf_y$ and, considering $\omega=nxdy-mydx$ we get $\omega\wedge df=-\alpha fdx\wedge dy$. So by the Saito's criterion, $\{nxdy-mydx,df\}$ is a Saito basis for $S(f)$.
\end{example}

Notice that a Saito basis for $S(f)$ is not unique. Indeed if $\{\omega_1, \omega_2\}$ is a Saito basis for $S(f)$ and $\alpha, \beta, \gamma, \delta\in\mathbb{C}$ with $\alpha\delta-\beta\gamma\neq 0$, then $\{\alpha\omega_1+\beta\omega_2, \gamma\omega_1+\delta\omega_2\}$ is also a Saito basis for $S(f)$.

If $\{\omega_1, \omega_2\}$ is a Saito basis for $S(f)$, we can easily express any element $\omega\in S(f)$ as $\omega=M_1\omega_1+M_2\omega_2$ with $M_i\in\mathbb{C}\{x,y\}$.

\begin{proposition}
Let $\{\omega_1,\omega_2\}$ be a Saito basis for $f$ such that $\omega_1\wedge\omega_2=ufdx\wedge dy$ with $u\in\mathbb{C}\{x,y\}$ a unit. If $\omega\in S(f)$ satisfies $\omega\wedge\omega_i=N_ifdx\wedge dy$ (see Lemma \ref{aux}), then $\omega=u^{-1}N_2\omega_1-u^{-1}N_1\omega_2$. 
\end{proposition}
\begin{proof} Let $\omega=Adx+Bdy, \omega_1=A_1dx+B_1dy$ and $\omega_2=A_2dx+B_2dy$.

Writing $\omega=q_1\omega_1+q_2\omega_2$, we obtain
$$N_1fdx\wedge dy=\omega\wedge\omega_1=q_2\omega_2\wedge\omega_1=-uq_2fdx\wedge dy$$
and
$$N_2fdx\wedge dy=\omega\wedge\omega_2=q_1\omega_1\wedge\omega_2=uq_1fdx\wedge dy.$$
Hence, $q_1=u^{-1}N_2,\ q_2=-u^{-1}N_1$ and  $\omega=u^{-1}N_2\omega_1-u^{-1}N_1\omega_2$. 
\end{proof}

By Remark \ref{remark-equi-saito}, if $\omega=Adx+Bdy\in S(f)$ then $Af_y-Bf_x=\frac{\omega\wedge df}{dx\wedge dy}= hf$, that is $h\in (J(f):f)$, where $J(f)=\langle f_x, f_y\rangle\subset \mathbb{C}\{x,y\}$ is the Jacobian ideal of $f$. Conversely, if $h\in (J(f):f)$, that is $hf=Af_y-Bf_x$, then we get $\omega=Adx+Bdy\in S(f)$. Therefore, we have $(J(f):f)=\{h;\ \omega\wedge df=hfdx\wedge dy\ \mbox{with}\ \omega\in S(f)\}$, which is called the {\bf cofactor ideal of $S(f)$}.  

\begin{lemma}\label{genideal} If $\{\omega_1,\omega_2\}$ is a Saito basis of $S(f)$ and $h_i\in\mathbb{C}\{x,y\}$ is the cofactor of $\omega_i$, then $(J(f):f)=\langle h_1,h_2\rangle$.
\end{lemma}
\begin{proof} It is immediate that $\langle h_1,h_2\rangle\subset (J(f):f)$.
	
	On the other hand, if $h\in (J(f):f)$, then there exist $A,B\in\mathbb{C}\{x,y\}$ such that  $hf=Af_x+Bf_y$. Considering $\omega=Bdx-Ady\in\Omega^1$, we get $\omega\wedge
	df=hfdx\wedge dy$, that is, $\omega\in S(f)$. Since $\{\omega_1,\omega_2\}$ is a Saito basis for $S(f)$, there exist $g_1, g_2\in\mathbb{C}\{x,y\}$ such that $\omega=g_1\omega_1+g_2\omega_2$ and thus, $hfdx\wedge
	dy=\omega\wedge df=(g_1h_1+g_2h_2)fdX\wedge dY$, which implies,
	$h=g_1h_1+g_2h_2\in\langle h_1,h_2\rangle$.
\end{proof}

The Saito module $S(f)$ is related to analytical invariants of $\mathcal{C}_f$, as we now describe.

Let $\mathcal{O}_f:=\frac{\mathbb{C}\{x,y\}}{\langle f\rangle}$ be the local ring of $\mathcal{C}_f$. The K\"ahler differential module of $\mathcal{O}_f$ is the $\mathcal{O}_f$-module
$\Omega_f:=\frac{\Omega^1}{F(f)}$ and its torsion submodule is $$\mathcal{T}_f:=\{\overline{\omega}\in\Omega_f;\ \overline{g}\overline{\omega}=\overline{0}\ \mbox{for some regular element}\ \overline{g}\in\mathcal{O}_f\}.$$ Consequently, we get $\mathcal{T}_f=\frac{S(f)}{F(f)}$.

By \cite[Theorem 1.1]{dimca}, we have
\begin{equation}\label{tjurina-torsion}
	\tau(f):=\dim_{\mathbb{C}}\frac{\mathbb{C}\{x,y\}}{\langle f,f_x,f_y\rangle}=\dim_{\mathbb{C}}\mathcal{T}_f=\dim_{\mathbb{C}}\frac{S(f)}{F(f)},\end{equation}
where $\tau(f)$ is the Tjurina number of $\mathcal{C}_f$, which is an analytic invariant. 

Notice that $\tau(f)\leq\mu(f)$, where $\mu(f):=\dim_{\mathbb{C}}\frac{\mathbb{C}\{x,y\}}{J(f)}$ is the Milnor number of $\mathcal{C}_f$, with equality if and only if $\mathcal{C}_f$ is analytically equivalent to a curve defined by a quasi-homogeneous polynomial (see \cite{saito1}).

\begin{remark}\label{equi-zero}
According to Berger (see \cite{berger}), the curve $\mathcal{C}_f$ is regular, that is, the multiplicity of $f$ is $m(f)=1$ if and only if $\mathcal{T}_f=\{0\}$, or equivalently, $0=\tau(f)=\dim_{\mathbb{C}}\mathcal{T}_f=\dim_{\mathbb{C}}\frac{S(f)}{F(f)}$, that is, $f$ is regular if and only if $S(f)=F(f)$. Moreover, if $f=\alpha x+\beta y+h$ where $h\in\langle x,y\rangle^2$, $\alpha ,\beta\in\mathbb{C}$ with $\alpha\neq 0$ or $\beta\neq 0$, then taking $g=\gamma x+\delta y$ such that $\alpha\delta -\beta\gamma\neq 0$, that is $\mathcal{C}_f$ and $\mathcal{C}_g$ are transversal curves, we obtain that $\{fdg,df\}$ is a Saito basis for $S(f)$, since
$$df\wedge fdg=(\alpha\delta-\beta\gamma+\delta h_x-\gamma h_y)fdx\wedge dy$$ and $\alpha\delta-\beta\delta+\delta h_x-\gamma h_y\in\mathbb{C}\{x,y\}$ is a unit.
\end{remark}

The following result was proved in the global case by Michler in \cite{michler} and used in the local case in \cite{transaction}. Since we are not aware of an explicit proof of this result for the local case, we provide one below.

\begin{proposition}\label{tjurina}
Let $\{\omega_1,\omega_2\}$ be a Saito Basis for $S(f)$ and $h_i\in\mathbb{C}\{x,y\}$ be the cofactor of $\omega_i$ for $i=1,2$. 
Then the $\mathbb{C}$-vector spaces $\frac{(J(f):f)}{J(f)}$ and $\frac{S(f)}{F(f)}=\mathcal{T}_f$ are isomorphic. In particular, we have 
$$\tau(f)=\dim_{\mathbb{C}}\frac{(J(f):f)}{J(f)}\ \ \ \mbox{and}\ \ \ \mu(f)-\tau(f)=I(h_1,h_2),$$
where $I(h_1,h_2):=\dim_{\mathbb{C}}\frac{\mathbb{C}\{x,y\}}{\langle h_1,h_2\rangle}$ denotes the intersection multiplicity of $h_1$ and $h_2$.
\end{proposition}
\begin{proof}
Firstly, notice that $\omega=Adx+Bdy\in F(f)$ if and only if it cofactor belongs to the Jacobian ideal $J(f)$. Indeed, if $\omega=Adx+Bdy\in F(f)=\mathbb{C}\{x,y\}df+f\Omega^1$ then there exist $P, Q, R\in\mathbb{C}\{x,y\}$ such that $A=Rf_x+fP$ and $B=Rf_y+fQ$. In this way, $\omega\wedge df=(Pf_y-Qf_x)fdx\wedge dy$ and the cofactor  $Pf_y-Qf_x$ belongs to $J(f)$. On the other hand, if $(Af_y-Bf_x)dx\wedge dy=\omega\wedge df=hfdx\wedge dy$ with $h=Pf_y-Qf_x
\in J(f)$ then $(fP-A)f_y=(fQ-B)f_x$. Since $f_x$ and $f_y$ are coprime, there exists $R\in\mathbb{C}\{x,y\}$ such that $fP-A=Rf_x$ and $fQ-B=Rf_y$ and consequently, $\omega=Adx+Bdy=(fP-Rf_x)dx+(fQ-Rf_y)dy=-Rdf+f(Pdx-Qdy)\in F(f)$.	
	
Now consider 
$$\begin{array}{cccc}
\varphi: & (J(f):f) & \rightarrow & \frac{S(f)}{F(f)} \\
& h & \mapsto & \overline{Adx+Bdy},\end{array}$$	
where $hf=Af_y-Bf_x$.

Notice that $\varphi$ is well defined. Indeed, if $A_1f_y-B_1f_x=hf=A_2f_y-B_2f_x$ then $(A_1-A_2)f_y=(B_1-B_2)f_x$. Since $f_x$ and $f_y$ are coprime, there exists $R\in\mathbb{C}\{x,y\}$ such that
$A_1=A_2+Rf_x$ and $B_1=B_2+Rf_y$, consequently
$\varphi(A_1f_y-B_1f_x)=\overline{A_1dx+B_1dy}=\overline{A_2dx+B_2dy+Rdf}=\overline{A_2dx+B_2dy}=\varphi(A_2f_y-B_2f_x)$.

Clearly, $\varphi$ is $\mathbb{C}$-linear and, as we have remarked $\ker(\varphi)=J(f)$. In addition, given $\overline{\omega}\in\frac{S(f)}{F(f)}$ we consider $h=\frac{\omega\wedge df}{fdx\wedge dy}\in (J(f):f)$ and we have $\varphi(f)=\overline{\omega}$, that is, $\varphi$ is surjective. It follows that $\frac{(J(f):f)}{J(f)}$ and $\frac{S(f)}{F(f)}=\mathcal{T}_f$ are $\mathbb{C}$-isomorphic.

In this way, $$\tau(f)=\dim_{\mathbb{C}}\frac{S(f)}{F(f)}=\dim_{\mathbb{C}}\frac{(J(f):f)}{J(f)}.$$
In addition, by Lemma \ref{genideal}, we get
$$\tau(f)=\dim_{\mathbb{C}}\frac{(J(f):f)}{J(f)}=\dim_{\mathbb{C}}\frac{\mathbb{C}\{x,y\}}{J(f)}-\dim_{\mathbb{C}}\frac{\mathbb{C}\{x,y\}}{(J(f):f)}=\dim_{\mathbb{C}}\frac{\mathbb{C}\{x,y\}}{J(f)}-\dim_{\mathbb{C}}\frac{\mathbb{C}\{x,y\}}{\langle h_1,h_2\rangle}.$$ 
Hence, $\tau(f)=\mu(f)-I(h_1,h_2)$.
\end{proof}

In Remark \ref{equi-zero}, we present an equivalent condition for $\mathcal{T}_f=\{0\}$ in terms of the Saito module. The following result provides equivalent conditions for $\mathcal{T}_f$ to be a nontrivial cyclic module. 

\begin{theorem}\label{theo-mu-tau}
Let $f\in\mathbb{C}\{x,y\}$ with $m(f)>1$. The following statements are equivalent:
\begin{enumerate}
	\item $\tau(f)=\mu(f)$.
	\item There exists a Saito basis $\{\omega ,df\}$ for $S(f)$.
	\item $\mathcal{T}_f$ is a nontrivial cyclic $\mathcal{O}_f$-module.
\end{enumerate}
\end{theorem}
\begin{proof}
($1\Leftrightarrow 2$) According to Proposition \ref{tjurina}, the condition $\tau(f)=\mu(f)$ is equivalent to the cofactor ideal of $S(f)$ being equal to $\mathbb{C}\{x,y\}$ that is, there exists $\omega\in S(f)$ such that $\omega\wedge df=ufdx\wedge dy$ with $u\in\mathbb{C}\{x,y\}$ a unit. By the Saito criterion, $\{\omega, df\}$ is a Saito basis for $S(f)$.

($2\Rightarrow 3$) If $\{\omega ,df\}$ is a Saito basis for $S(f)$, then $\{\overline{\omega}\}$ is a generating set for $\mathcal{T}_f=\frac{S(f)}{F(f)}$ with $\overline{\omega}\neq\overline{0}$, since otherwise we would have $S(f)=F(f)$, contradicting the assumption that $m(f)>1$ (see Remark \ref{equi-zero}).

($3\Rightarrow 2$) If $\mathcal{T}_f=\frac{S(f)}{F(f)}$ is cyclic and generated by $\{\overline{\omega}\}$ with $\omega=Adx+Bdy\not\in F(f)$, then there exists $\omega_1=h_1df+f(Cdx+Ddy)\in F(f)$, such that $\{\omega, \omega_1\}$ is a Saito basis for $S(f)$. By the Saito criterion, there are $h\in\mathbb{C}\{x,y\}$ and a unit $u\in\mathbb{C}\{x,y\}$ such that $\omega\wedge df=hfdx\wedge dy$ and $$ufdx\wedge dy=\omega\wedge\omega_1=f(hh_1+AD-BC)dx\wedge dy,$$ so $u=hh_1+AD-BC$. Since $A$ and $B$ are not units, the term $A D - B C$ is not a unit either, implying that $h h_1$ must be a unit. Therefore, $h$ is a unit and the set $\{\omega ,df\}$ satisfies the Saito criterion, that is, it is a Saito basis for $S(f)$.
\end{proof}

\begin{example}\label{m-2} Let us consider $f=y^2+A(x)y+B(x)\in\mathbb{C}\{x\}[y]$ with $A,B\in\mathbb{C}\{x\}$ and $m(f)=2$, that is $mult(A)\geq 1$ and $mult(B)\geq 2$. Since $m(f)=2$, we get $\tau(f)=\mu(f)$ and $\mathcal{C}_f$ is analytically equivalent to a curve defined by a quasi-homogeneous polynomial. Indeed, setting $z=y+\frac{A}{2}$ and $w=\left ( \frac{A^2}{4}-B\right )^\frac{1}{k}$, where $$\frac{A^2}{4}-B=x^k u(x)$$ with $k\geq 2$ and $u(x)\in\mathbb{C}\{x\}$ a unit, we get
	$f(x,y)=g(w,z)=z^2-w^k.$
By Example \ref{qh}, $\{2wdz-kzdw, dg=2zdz-kw^{k-1}dw\}$ is a Saito basis for $S(g)$. Consequently, writing $A'=\frac{dA}{dx}$ and $B'=\frac{dB}{dx}$, we obtain: 
$$\begin{array}{ll}
	\omega& =2wdz-kzdw= 2\left ( \frac{A^2}{4}-B\right )^{\frac{1}{k}}d\left (y+\frac{A}{2}\right )-k\left (y+\frac{A}{2}\right )d\left ( \left (\frac{A^2}{4}-B\right )^{\frac{1}{k}}\right )\\
	& =2\left ( \frac{A^2}{4}-B\right )^{\frac{1}{k}}dy+\left (\frac{A^2}{4}-B\right )^{\frac{1}{k}}\left ( A'-\left ( y+\frac{A}{2}\right )\left ( \frac{AA'}{2}-B'\right )\left (\frac{A^2}{4}-B\right )^{-1}\right )dx\\
	& =
	u^{\frac{1}{k}}(x)\left (2xdy-\left ( \left (k+\frac{xu'(x)}{u(x)}\right )\left ( y+\frac{A}{2}\right)-xA'\right )dx \right ),\end{array}$$
where the last equality follows because $\frac{A^2}{4}-B=x^k u(x)$ with $k\geq 2$ and
\begin{equation}\label{conta}\left ( \frac{A^2}{4}-B\right )^{\frac{1}{k}-1}\left ( \frac{AA'}{2}-B'\right )=(x^ku(x))^{\frac{1-k}{k}}(kx^{k-1}u(x)+x^ku'(x))=u(x)^{\frac{1}{k}}\left (k+\frac{xu'(x)}{u(x)}\right ).\end{equation}
Define $\omega_1=u^{-\frac{1}{k}}(x)\omega=2xdy-\left ( \left (k+\frac{xu'(x)}{u(x)}\right )\left ( y+\frac{A}{2}\right)-xA'\right )dx$	
and
$$\begin{array}{ll}
	\omega_2 & =df =2\left ( y+\frac{A}{2}\right )\left ( dy+\frac{A'}{2}dx\right )-k\left ( \frac{A^2}{4}-B\right )^{\frac{k-1}{k}}\left ( \frac{1}{k}\left ( \frac{A^2}{4}-B\right )^{\frac{1-k}{k}}\left ( \frac{AA'}{2}-B'\right )\right )dx \\
	& = (2y+A)dy+(A'y+B')dx.
\end{array}$$	
We have that
$$\omega_1\wedge\omega_2 =-2\left ( \left (y+\frac{A}{2}\right )^2\left ( k+\frac{xu'(x)}{u(x)}\right )-x\left ( \frac{AA'}{2}-B'\right )\right )dx\wedge dy= -2 \left ( k+\frac{xu'(x)}{u(x)}\right )fdx\wedge dy
$$	where the last equality follows by (\ref{conta}).
Consequently, $\{\omega_1,\omega_2\}$ is a Saito Basis for $S(f)$.
\end{example}

Besides the Tjurina number, other analytic invariants can also be obtained via the Saito module, as we describe below.

Let $f = \prod_{i=1}^{r}f_i$, where  $f_i\in\mathbb{C}\{x,y\}$ is irreducible and $\langle f_i\rangle\neq\langle f_j\rangle$ for $i\neq j$. Each $f_i$ defines an irreducible plane curve (a branch) $\mathcal{C}_i$ that admits a primitive\footnote{A parameterization is called primitive if all the exponents have no common factor greater than one.} parametrization $\varphi_i(t_i) = (x_i, y_i)\in\mathbb{C}\{t_i\}\times\mathbb{C}\{t_i\}$.

For any $h\in\mathbb{C}\{x,y\}$, we denote by $\varphi_i^*(h)=h(x_i,y_i)\in\mathbb{C}\{t_i\}$, and obtain the exact sequence
$$0\rightarrow \langle f_i\rangle \rightarrow \mathbb{C}\{x,y\}\stackrel{\varphi_i^*}{\rightarrow}\mathbb{C}\{x_i,y_i\}\rightarrow 0.$$
Consequently, $\mathcal{O}_i:=\frac{\mathbb{C}\{x,y\}}{\langle f_i\rangle}\cong \mathbb{C}\{x_i,y_i\}\subseteq\mathbb{C}\{t_i\}$. In this way, the field of fraction of $\mathcal{O}_i$ is $\mathcal{K}_i\cong \mathbb{C}(t_i)$ and we have a natural discrete valuation
$$\begin{array}{cccc}
\nu_i: & \mathcal{K}_i & \rightarrow & \overline{\mathbb{Z}}:=\mathbb{Z}\cup \{\infty\} \\
 & \frac{p}{q} & \mapsto & ord_{t_i}(p)-ord_{t_i}(q),\end{array}$$
 where $\nu_i(0)=\infty$.
In particular, $\Gamma_i:=\nu_i(\mathcal{O}_i)\subseteq\mathbb{N}\cup\{\infty\}$ is the semigroup associated with the branch $\mathcal{C}_i$ and it is a complete topological invariant of $\mathcal{C}_i$.

Since the total ring of fraction of $\mathcal{O}_f$ is $\mathcal{K}_f=\bigoplus_{i=1}^{r}\mathcal{K}_i\cong\bigoplus_{i=1}^{r}\mathbb{C}(t_i)$ we can consider
 $$\begin{array}{cccc}
 	\nu_f: & \mathcal{K}_f & \rightarrow & \overline{\mathbb{Z}}^r \\
 	& \frac{p}{q} & \mapsto & \left ( \ldots ,\nu_i\left( \pi_i\left  (\frac{p}{q} \right )\right ),\ldots \right ),\end{array}$$
where $\pi_i:\mathcal{K}_f\rightarrow \mathcal{K}_i$ is the natural projection.

As before, we define $\Gamma(f):=\nu_f(\mathcal{O}_f)$ as the value semigroup of $\mathcal{C}_f$. This semigroup determines, and it is determined by, the topological class of $\mathcal{C}_f$ (see \cite{Delgado}, \cite{Waldi} and \cite{Z-top}). In contrast to $\Gamma_i$, the set $\Gamma(f)$ is not a finitely generated semigroup, but $(\Gamma(f), \inf ,+)$, that is, equiped with the tropical operations, it is a finitely generated semiring\footnote{$(\Gamma(f),\inf +)$ is a semiring means that $(\Gamma(f), \inf)$ and $(\Gamma(f),+)$ are monoids with identity elements $(\infty,\ldots ,\infty)$ and $(0,\ldots ,0)$ respectively; $\inf\{\alpha + \beta,\alpha +\gamma\}=\alpha+\inf\{\beta,\gamma\}$ and $\infty +\alpha =\infty$ for every $\alpha, \beta, \gamma\in\Gamma(f)$.} (see \cite{carvalho1}).

Considering the canonical epimorphism $\rho:\Omega_f\rightarrow\frac{\Omega_f}{\mathcal{T}_f}$ and the $\mathcal{O}_f$-monomorphism (see Section 2 in \cite{BGHH}) $\psi:\frac{\Omega_f}{\mathcal{T}_f}\rightarrow\bigoplus_{i=1}^{n}\mathbb{C}\{t_i\}\subset\mathcal{K}_f$  defined by
\begin{equation}\label{psi} \psi(\overline{Adx+Bdy})=\left (\ldots ,t_i(\varphi_i^*(A)x'_i+\varphi_i^*(B)y'_i),\ldots \right )\end{equation}
we have 
$$\Lambda(f):=\{\nu_f(\omega):=\nu_f(\psi\circ\rho(\omega));\ \omega\in\Omega_f\}$$
the values set of the K\"ahler differentials. The set $\Lambda(f)$ satisfies $\Gamma(f)\setminus\{(0,\ldots ,0)\}\subseteq\Lambda(f)$ and $\Gamma(f)+\Lambda(f)\subseteq\Lambda(f)$. In addition, $\Lambda(f)$ is an analytic invariant of $\mathcal{C}_f$ and it plays a central role in the solution of the analytic classification problem for plane curves as presented in \cite{HR}.

On the other hand, given $\omega\in S(f)$ such that $g\omega =kdf+f\eta$ with $g, k\in\mathbb{C}\{x,y\}$, $gcd(g,f)=1$ and $\eta\in\Omega^1$, Saito (see \cite{saito}) shows that $res_f(\omega):=\overline{\frac{k}{g}}\in\mathcal{K}_f$ is well defined, where the bar denotes the class in $\mathcal{O}_f$, and
$$\mathcal{R}(f):=\left \{ res_f(\omega);\ \omega\in S(f)\right \}$$
is an $\mathcal{O}_f$-submodule of $\mathcal{K}_f$ generated by $res_f(\omega_1)$ and $res_f(\omega_2)$, where $\{\omega_1,\omega_2\}$ is a Saito basis for  $S(f)$.

Pol, in \cite[Corollary 3.32]{Pol}, shows that 
$$\Delta(f):=\{\nu_f(res_f(\omega));\ \omega\in S(f)\}\subseteq \overline{\mathbb{Z}}^r$$ determines and is determined by $\Lambda(f)$. Consequently, it is also an analytic invariant of $\mathcal{C}_f$.

\begin{remark}
Notice that if $\omega=Adx+Bdy\in S(f)$ then $(Af_y-Bf_x)dx\wedge dy=\omega\wedge df=hfdx\wedge dy$. So,
$$f_x\omega=A(df-f_ydy)+Bf_xdy=Adf-hfdy\ \ \mbox{and}\ \ f_y\omega=Af_ydx+B(df-f_xdx)=Bdf+hfdx,$$ that is,
$res(\omega)=\overline{\frac{A}{f_x}}=\overline{\frac{B}{f_y}}$ and \begin{equation}\label{GSV}
\nu_f(res(\omega))=\nu_f(A)-\nu_f(f_x)=\nu_f(B)-\nu_f(f_y).\end{equation}
\end{remark}

If $\omega=Adx+Bdy\in S(f)$ satisfies $gcd(A,B)=1$, then the $1$-form  $\omega$ defines a holomorphic foliation $\mathcal{F}_{\omega}$ in $(\mathbb{C}^2,0)$ and the condition $\omega\wedge df=hfdx\wedge dy$ implies that $\mathcal{C}_f$ is an invariant curve (a separatrix) of $\mathcal{F}_{\omega}$ passing through  $0\in\mathbb{C}^2$.

\begin{remark}\label{foliation-irred}
	If $f\in\mathbb{C}\{x,y\}$ is irreducible with $m(f)>1$ and $\{\omega_1=A_1dx+B_1dy,\omega_2=A_2dx+B_2dy\}$ is a Saito basis for $S(f)$, then  $\gcd(A_i,B_i)=1$.
	
	Indeed, if $\gcd(A_1,B_1)=g\in\langle x,y\rangle$, then $\omega_1=g\eta_1$ and, by Saito's criterion, there exists a unit $u\in\mathbb{C}\{x,y\}$ such that $ufdx\wedge dy=\omega_1\wedge\omega_2=g\eta_1\wedge\omega_2$. Since $f$ is irreducible and $g\in\langle x,y\rangle$, we get $g=vf$ for some unit $v\in\mathbb{C}\{x,y\}$. Thus, $\omega_1=g\eta_1=vf\eta_1\in F(f)$, so the torsion submodule $\mathcal{T}_f=\{0\}$ or it is cyclic. If $\mathcal{T}_f=\{0\}$, then by Remark \ref{equi-zero} we get $m(f)=0$, that is a contradiction. If $\mathcal{T}_f$ is cyclic, then by Theorem \ref{theo-mu-tau}, $\{\omega_1,df\}=\{f\eta_1,df\}$ would be a Saito basis for $S(f)$, implying again that $\mathcal{T}_f=\{0\}$, which is an absurd.	
\end{remark}

In what follows, in this section, we will assume that  $f\in\mathbb{C}\{x,y\}$ is irreducible. 

By the previous remark, given a Saito basis $\{\omega_1,\omega_2\}$ for $S(f)$, each $1$-form  $\omega_i$ defines a holomorphic foliation $\mathcal{F}_i$ in $(\mathbb{C}^2,0)$ for $i=1,2$, and some authors have studied the generators of $S(f)$ from this perspective (see \cite{felipe-nuria-david}, \cite{yohann}, \cite{yohann-marcelo}, etc.). 

Let us present some properties of $S(f)$ and the foliation $\mathcal{F}_i$ defined by $\omega_i$ for $i=1,2$. 

If $\{\omega_1=A_1dx+B_1dy,\omega_2=A_2dx+B_2dy\}$ is a Saito basis for $S(f)$, then according to (\ref{GSV}), the GSV-index $GSV(\mathcal{F}_i,\mathcal{C}_f,0)$ (see \cite{brunela} for details) of $\mathcal{F}_i$ with respect to $\mathcal{C}_f$ at the origin coincides with the value of $res_f(\omega_i)$, that is,
\begin{equation}\label{gsv}GSV(\mathcal{F}_i,\mathcal{C}_f,0)=\nu_f(res_f(\omega_i)).\end{equation}

The GSV-index of a foliation $\mathcal{F}$ provides some of its geometric properties. For instance, by \cite[Proposition 6]{brunela}, if $\mathcal{C}$ is a separatrix of $\mathcal{F}$ and $GSV(\mathcal{F},\mathcal{C},0)<0$, then $\mathcal{F}$ is dicritical, that is, it admits infinitely many separatrices through $0\in\mathbb{C}^2$.

The next proposition gives us an intrinsic property of a Saito basis from the Foliation Theory view point.  It was firstly obtained by Genzmer (see \cite{yohann} and \cite{yohann2}), but here we present an alternative proof.

\begin{proposition}
Let $f\in\mathbb{C}\{x,y\}$ be irreducible with $m(f)>1$. Given any Saito basis $\{\omega_1=A_1dx+B_2dy,\omega_2=A_2dx+B_2dy\}$ for $S(f)$, we have $\nu_f(B_1)=\min\{\nu_f(B_1),\nu_f(B_2)\}<\nu_f(f_y)$. In particular, $\omega_1$ defines a dicritical foliation $\mathcal{F}_1$ at the origin.
\end{proposition} 
\begin{proof}
Given any Saito basis $\{\omega_1=A_1dx+B_1dy,\omega_2=A_2dx+B_2dy\}$, we may assume that $\nu_f(B_1)=\min\{\nu_f(B_1),\nu_f(B_2)\}$. Then, by (\ref{GSV}), we get $\nu_f(res_f(\omega_1))=\min\{\nu(res_f(\omega));\ \omega\in S(f)\}=\min\Delta(f)$. By \cite[Prop. 3.21]{Pol}, we have $\min\Delta(f)\leq -m(f)+1$, and since $m(f)>1$, it follows from (\ref{gsv}) that
$GSV(\mathcal{F}_1,\mathcal{C}_f,0)=\nu_f(res_f(\omega_1))<0$, or equivalently, $\nu_f(B_1)<\nu_f(f_y)$. Hence, by \cite[Proposition 6]{brunela},  we conclude that $\mathcal{F}_1$ is a dicritical foliation.
\end{proof}

\section{Computing a Saito Basis for $S(f)$}

In the previous section, we presented some properties of a Saito basis for \(S(f)\) and its relation to  analytic invariants of \(\mathcal{C}_f\). In Example \ref{qh} and Remark \ref{equi-zero}, we provide a Saito bases for some particular cases. 

Computing a Saito basis is not an easy task. Several algorithms are available to compute it, as shown in \cite{tajima1}, \cite{tajima2}, \cite{walcher}, among others. Since \(S(f)\) is closely related to the torsion submodule \(\mathcal{T}_f\) of \(\Omega_f\), i.e., \(\mathcal{T}_f=\frac{S(f)}{F(f)}\), we will present a method for computing a Saito basis for \(S(f)\) by using a standard basis for \(\frac{\Omega_f}{\mathcal{T}_f}\). Algorithms for computing a standard basis for \(\Omega_f\) are presented in \cite{carvalho} and \cite{basestandard}. For the reader's convenience, we outline the main ingredients of these algorithms. 

To simplify the presentation, we assume that \(f\in\mathbb{C}\{x,y\}\) is irreducible with $m(f)>1$.

Since $\Omega_f=\frac{\Omega^1}{F(f)}$ and $\mathcal{T}_f=\frac{S(f)}{F(f)}$, we get 
$\frac{\Omega_f}{\mathcal{T}_f}\cong\frac{\Omega^1}{S(f)}$.
Considering the natural $\mathcal{O}_f$-module structue we have the homomorphisms
$$F(f)\ \hookrightarrow\ S(f)\ \hookrightarrow\ \Omega^1\ \stackrel{\sigma}{\twoheadrightarrow}\  \frac{\Omega^1}{F(f)}=\Omega_f\ \stackrel{\rho}{\twoheadrightarrow}\  \frac{\Omega^1}{S(f)}\cong\frac{\Omega_f}{\mathcal{T}_f} \ \stackrel{\psi}{\hookrightarrow}\ \mathcal{K}_f$$
where $\sigma$ and $\rho$ are the canonical epimorphisms and $\psi$ is the monomorphism defined in (\ref{psi}).

Since $f\in\mathbb{C}\{x,y\}$ is irreducible, we have that $\Gamma(f)=\nu_f(\mathcal{O}_f)$ is a numerical semigroup, that is, it admits a minimal finite set of generators $0<v_0<v_1\ldots <v_g$ such that $\Gamma(f)=v_0\mathbb{N}+\ldots +v_g\mathbb{N}$, and there exists $c=\min\{\gamma\in\Gamma(f);\ \gamma+\mathbb{N}\subset\Gamma(f)\}$, which coincides with the Milnor number $\mu(f)$. 

A {\bf Standard basis} for $\frac{\Omega^1}{S(f)}$ is a finite set $\mathcal{B}\subset\frac{\Omega^1}{S(f)}$ such that for any $0\neq\overline{\eta}\in\frac{\Omega^1}{S(f)}$, there exist $\overline{\omega}\in \mathcal{B}$ and $\overline{h}\in\mathcal{O}_f$ such that $\nu_f(\overline{\eta})=\nu_f(\overline{h\omega})=\nu_f(\overline{h})+\nu_f(\overline{\omega})$, where $\overline{\omega}$ denotes the class of $\omega\in\Omega^1$ in $\frac{\Omega^1}{S(f)}$ and $\overline{h}$ is the class of $h\in\mathbb{C}\{x,y\}$ in $\mathcal{O}_f=\frac{\mathbb{C}\{x,y\}}{\langle f\rangle}$. 
In particular, we have $$\Lambda(f)=\bigcup_{\overline{\omega}\in \mathcal{B}}\left (\nu_f(\overline{\omega})+\Gamma(f)\right ).$$

Let $\varphi(t)=(x(t),y(t))\in\mathbb{C}\{t\}\times\mathbb{C}\{t\}$ be a primitive parametrization for $\mathcal{C}_f$. Up to a change of analytic coordinates, we may assume $\nu_f(\varphi^*(x))=v_0<v_1=\nu_f(\varphi^*(y))$ with $m(f)=v_0<v_1$ and $v_0\nmid v_1$. In addition, we set $f_i\in\mathbb{C}\{x,y\}$ such that $\nu_f(f_i):=ord_t\varphi^*(f_i)=v_i$ for $0\leq i\leq g$. In particular, we take $f_0=x$ and $f_1=y$. 

Given $\emptyset\neq Q\subset \frac{\Omega^1}{S(f)}$, we say that $r$ is a {\it reduction} of $0\neq\overline{\omega}\in \frac{\Omega^1}{S(f)}$ modulo $Q$ if there exist $\alpha_i\in\mathbb{N}$ and $q\in Q$ such that
$$r=\overline{\omega}-\prod_{i=0}^{g}\overline{f}_i^{\alpha_i}q, \ \ \mbox{with}\ \ \nu_f(r)>\nu_f(\overline{\omega}).$$

We say that
$r$ is a {\bf final reduction} of $\overline{\omega}$ modulo $Q$ (and write $\overline{\omega} \;
\stackrel{Q}{\longrightarrow} \; r$) when $r$ is obtained from $\overline{\omega}$ through a (possibly infinite) chain of
reductions modulo $Q$ and cannot be reduced further.

A {\bf minimal $S$-process} of a pair of elements
$\overline{\omega},\overline{\eta}\in\frac{\Omega^1}{S(f)}$ is defined as
$$s(\overline{\omega},\overline{\eta})=a\prod_{i=0}^{g}\overline{f}_i^{\alpha_i}\overline{\omega}+b\prod_{i=0}^{g}\overline{f}_i^{\beta_i}\overline{\eta},$$
where $a,b\in \mathbb{C}$ and
$$\nu_f\left (s(\overline{\omega},\overline{\eta})\right )>\nu_f\left (\prod_{i=0}^{g}\overline{f}_i^{\alpha_i}\overline{\omega}\right)=\nu_f\left (\prod_{i=0}^{g}\overline{f}_i^{\beta_i}\overline{\eta}\right )$$ whenever $s(\overline{\omega},\overline{\eta})\neq 0$. Here, $(\alpha_0,\ldots ,\alpha_{g},\beta_0,\ldots
,\beta_{g})$ is a minimal solution of the linear Diophantine equation
$\sum_{i=0}^{g}\alpha_iv_i+\nu_f(\overline{\omega})=\sum_{i=0}^{g}\beta_iv_i+\nu_f(\overline{\eta})$.

For each pair  $\overline{\omega},\overline{\eta}\in\frac{\Omega^1}{S(f)}$, there exists a finite
number $n(\overline{\omega},\overline{\eta})$ of minimal $S$-processes, which we denote by
$s_i(\overline{\omega},\overline{\eta})$ with $i=1,\ldots ,n(\overline{\omega},\overline{\eta})$.

In the sequel we present the Algorithm $4.10$ from \cite{basestandard}, which allows one to
compute a Standard basis for $\frac{\Omega^1}{S(f)}$ in a finite number of steps.

\newpage
\begin{center} {\bf Algorithm 1:} Standard basis $\mathcal{B}$ for
$\frac{\Omega^1}{S(f)}$: \vspace{0.2cm}

\begin{tabular}{|l|}
	\hline
\noindent{\bf input:} $f_i\ \mbox{such that}\ \nu_f(f_i)=v_i\ \mbox{for}\ i=0,\dots ,g$;\\

\noindent{\bf define:} $Q_{-1}=\emptyset$,\ $Q_0=\{\overline{df_i},\ i=0,\dots ,g\}$ and $k:=0;$\\

\noindent{\bf while $Q_{k-1}\neq Q_k$ do}\\

\hspace{0.3cm} $S:= \{ s(\overline{\omega},\overline{\eta});\ s(\overline{\omega},\overline{\eta})$ is a minimal $S$-process of $\overline{\omega},\overline{\eta}\in Q_k\};$\\

\hspace{0.3cm} $R:= \{ r;\ s(\overline{\omega},\overline{\eta}) \; \stackrel{{Q_k}}{\longrightarrow}\; r$ and $r\not =0,\;
\forall s(\overline{\omega},\overline{\eta})\in S \};$\\

\hspace{0.3cm} $Q_{k+1}:=Q_i\cup R;$\\

\hspace{0.3cm} $k:=k+1;$\\

\noindent {\bf output:} $\mathcal{B}=Q_k.$\\
\hline
\end{tabular}
\end{center}

\vspace{0.2cm}

\begin{example}\label{ex-sb} Consider $gcd(n,m)=1<n<m$ with $n\neq 2$
		and the curve $\mathcal{C}_f$ defined by $$f=y^n-x^m+x^{m-k}y^{n-2}\ \ \mbox{with}\ \ 
		2\leq k< min\left \{ \frac{(n-2)m}{n},2+\frac{2m}{n}\right \}.$$
		
Notice that $f\in \mathbb{C}\{x,y\}$ is irreducible and it admits a parametrization $$\varphi(t)=\left (t^n,t^m-\frac{1}{n}t^{(n-1)m-kn}+\sum_{i>(n-1)m-kn}a_it^i\right ).$$ Since $k\leq \frac{(n-2)m}{n}$, we obtain $m<(n-1)m-kn$, and therefore $\Gamma(f)=n\mathbb{N}+m\mathbb{N}$ and $\mu(f)=(n-1)(m-1)=\min\{\gamma\in\Gamma(f);\ \gamma+\mathbb{N}\subset\Gamma(f)\}$.
		
Let us apply Algorithm 1 to compute a Standard
		basis for $\frac{\Omega^1}{S(f)}$ and the values set $\Lambda(f)$.
		
		Since $\nu_f(x)=n$ and $\nu_f(y)=m$, we set $f_0=x$ and $f_1=y$.
		
		{\bf Step $i=0$:} We have		$Q_0=\{\overline{dx},\overline{dy}\}$ and 			$$S=\{s_1(\overline{dx},\overline{dy})=n\overline{x}\overline{dy}-m\overline{y}\overline{dx},\ s_2(\overline{dx},\overline{dy})=n\overline{y}^{n-1}\overline{dy}-m\overline{x}^{m-1}\overline{dx}\}.$$
		From (\ref{psi}), we get			$$s_1(\overline{dx},\overline{dy})=n\overline{x}\overline{dy}-m\overline{y}\overline{dx}=-((n-2)m-kn)t^{(n-1)m-(k-1)n}+\sum_{i>(n-1)m-kn}(i-m)nt^{i+n}$$
		and
			$$s_2(\overline{dx},\overline{dy})+(m-k)\overline{x}^{m-k-1}\overline{y}^{n-2}\overline{dx}+(n-2)\overline{x}^{m-k}\overline{y}^{n-3}\overline{dy}=\overline{df}=0,$$ that is, the final reduction of		$s_2(\overline{dx},\overline{dy})$ modulo $Q_0$ is zero. 
		
		On the other hand, since 		$\nu_f(s_1(\overline{dx},\overline{dy}))=(n-1)m-(k-1)n\not\in\Gamma(f)$, we have that $s_1(\overline{dx},\overline{dy})$ coincides with its final reduction modulo $Q_0$. In addition, we get $\nu_f(\overline{x}^rs_1(\overline{dx},\overline{dy}))=(n-1)m-(k-1-r)n\not\in\Gamma(f)$ for $0\leq r\leq k-2$, so we have $\{(n-1)m-(k-1-r)n;\ 0\leq r\leq k-2\}\subseteq \Lambda(f)\setminus\Gamma(f)$ and any integer greater than $(n-2)m-n$ belongs to $\Lambda(f)$.
		
		{\bf Step $i=1$:} In this step we have 		$Q_1=\{\overline{dx},\overline{dy},\overline{\omega}_1:=n\overline{x}\overline{dy}-m\overline{y}\overline{dx}\}$ and
		$$\begin{array}{ll}			S =  \left \{ \right . & s_1(\overline{dx},\overline{dy}),\ 			s_2(\overline{dx},\overline{dy}),\vspace{0.15cm} \\
			& s_1(\overline{dx},\overline{\omega}_1)=\frac{(n-2)m-kn}{n}\overline{x}^{m-k}\overline{dx}+\overline{y}\overline{\omega}_1,			\  s_2(\overline{dx},\overline{\omega}_1)=\frac{(n-2)m-kn}{n}\overline{y}^{n-1}\overline{dx}+\overline{x}^k\overline{\omega}_1, \vspace{0.15cm} \\ 
			 & \left . s_1(\overline{dy},\overline{\omega}_1)=\frac{(n-2)m-kn}{m}\overline{y}^{n-2}\overline{dy}+\overline{x}^{k-1}\overline{\omega}_1,\  s_2(\overline{dy},\overline{\omega}_1)=\frac{(n-2)m-kn}{m}\overline{x}^{m-k+1}\overline{dy}+\overline{y}^2\overline{\omega}_1\ \			\right \}.\end{array}$$
		
		It is not necessary to analyze the first two $S$-processes, as they were already considered in the previous step. Since $k<2+\frac{2m}{n}$, we get 	$$\min\{\nu_f(s_1(\overline{dx},\overline{\omega}_1)),\nu_f(s_2(\overline{dx},\overline{\omega}_1)),\nu_f(s_1(\overline{dy},\overline{\omega}_1)),\nu_f(s_2(\overline{dx},\overline{\omega}_1))\}>(m-k+1)n>(n-2)m-n$$ and consequently, each $S$-process in $S$ has a vanishing final reduction modulo $Q_1$. Hence, $\mathcal{B}=Q_1=\{\overline{dx}, \overline{dy}, \overline{\omega_1}=n\overline{xdy}-m\overline{ydx}\}$ is a Standard basis for $\frac{\Omega^1}{S(f)}$ and \begin{equation}\label{r}\Lambda(f)\setminus\Gamma(f)= \{(n-1)m-(k-1-r)n;\ 	0\leq r\leq k-2\}.\end{equation}
\end{example}

\begin{remark}
By \cite[Proposition 3.31]{Pol} (or \cite[Corollary 6]{BGHH}), there is an $\mathcal{O}_f$-isomorphism
$$\begin{array}{cccc}
	\Phi : & J(f)\mathcal{O}_f:=\mathcal{O}_ff_x+\mathcal{O}_ff_y & \rightarrow & \frac{\Omega_f}{\mathcal{T}_f}\cong\frac{\Omega^1}{S(f)} \\
	& \overline{Bf_x+Af_y} & \mapsto & \overline{Adx-Bdy}
\end{array}$$
satisfying $\nu_f(\overline{Adx-Bdy})+\mu(f)-1=\nu_f(\overline{Bf_x+Af_y})=I(f,Bf_x+Af_y)$. In this way, we can apply the Algorithm 1 to obtain a Standard basis for $J(f)\mathcal{O}_f$.
\end{remark}

If $\mathcal{B}=\{\overline{\eta}_1,\ldots ,\overline{\eta}_s\}$ is a Standard basis for $\frac{\Omega^1}{S(f)}$, then, by Algorithm 1, every minimal
$S$-process of a pair of elements in $\mathcal{B}$ has a zero final reduction modulo $\mathcal{B}$, that is,
\begin{equation}\label{reduc}
	s_l(\overline{\eta}_j,\overline{\eta}_k)=\left (a\overline{f}_0^{\alpha_{jkl0}}\cdot\ldots\cdot\overline{f}_g^{\alpha_{jklg}}\right )\overline{\eta}_j-\left (b\overline{f}_0^{\beta_{jkl0}}\cdot\ldots\cdot\overline{f}_g^{\beta_{jklg}}\right )\overline{\eta}_k=\sum_{i=1}^{s}\overline{Q}_{ijkl}\overline{\eta}_i\end{equation}
where $Q_{ijkl}\in\mathbb{C}\{f_0,\ldots ,f_g\}$ for
$1\leq i,j,k\leq s$ and $l=1,\ldots ,n(j,k)$. 

Define
\begin{equation}\label{E}
\overline{E}_{ijkl}=\left \{ 
\begin{array}{ll}
	\overline{Q}_{jjkl}-a\overline{f}_0^{\alpha_{jkl0}}\cdot\ldots\cdot\overline{f}_g^{\alpha_{jklg}} & \mbox{if}\ \ i=j\vspace{0.15cm}\\
		\overline{Q}_{kjkl}+b\overline{f}_0^{\beta_{jkl0}}\cdot\ldots\cdot\overline{f}_g^{\beta_{jklg}} & \mbox{if}\ \ i=k\vspace{0.15cm}\\
		\overline{Q}_{ijkl} & \mbox{otherwise,}
\end{array}\right .
\end{equation}
so, we have $\sum_{i=1}^{s}\overline{E}_{ijkl}\overline{\eta}_i=\overline{0}$, that is, $\sum_{i=1}^{s}E_{ijkl}\eta_i \in S(f)$.

Notice that $(\overline{E}_{1jkl},\ldots ,\overline{E}_{sjkl})$ is a syzygy for $(\overline{\eta}_1,\ldots,\overline{\eta}_r)$. By the Schreyer's Theorem (see Theorem 2.5.9 in \cite{singular}), the set  $\{(\overline{E}_{1jkl},\ldots ,\overline{E}_{sjkl});\ 1\leq j,k\leq s\ \mbox{and}\ 1\leq l\leq n(j,k)\}$ generates the $\mathcal{O}_f$ of syzygies. Consequently,
 $$\left \{\sum_{i=1}^{s}E_{ijkl}\eta_i;\ 1\leq j,k\leq s\ \mbox{and}\ 1\leq l\leq n(j,k)\right \}$$
 is a generating set for $S(f)$. 
 
 This leads to the following procedure to compute a Saito basis for $S(f)$:

\vspace{0.2cm}
\begin{center} {\bf Algorithm 2:} Saito basis for $S(f)$: \vspace{0.2cm}
	
	\begin{tabular}{|l|}
		\hline
		\noindent{\bf Input:} $f_i$ such that $\nu_f(f_i)=v_i$ for $i=0,\dots ,g$;\\
		
		\noindent{\bf Compute:} \\
		
		\hspace{0.5cm} $\mathcal{B}:=\{\overline{\eta}_1,\ldots ,\overline{\eta}_s\}$ a Standard Basis for $\frac{\Omega^1}{S(f)}$ via Algorithm 1; \\
		
		\hspace{0.5cm} $S:= \{$ minimal $S$-process\ $s_l(\overline{\eta_j},\overline{\eta_k})$ for every  $\overline{\eta}_j,\overline{\eta}_k\in \mathcal{B}\};$\\
		
		\noindent{\bf For each} $s_{l}(\overline{\eta}_j,\overline{\eta}_k)\in S$\  {\bf compute its reduction modulo $\mathcal{B}$ as (\ref{reduc})} \\
		\noindent{\bf and determine $\overline{E}_{ijkl}$ for
		all $1\leq i,j,k\leq s$ and $1\leq l\leq n(j,k)$ as (\ref{E});}\\
		
		\noindent{\bf Define} $G:=\{\sum_{i=1}^{s}E_{ijkl}\eta_i;\ E_{ijkl}\eta_i$ is a representative of $\overline{E}_{ijkl}\overline{\eta}_i\}$;\\
		
		\noindent{\bf Select $\omega_1,\omega_2\in G$ such that $\omega_1\wedge\omega_2=ufdx\wedge dy$ for some unit $u\in\mathbb{C}\{x,y\}$};\\
		
		\noindent {\bf Output:} $\{\omega_1,\omega_2\}$ a Saito basis for $S(f)$.\\
		\hline
	\end{tabular}
\end{center}

\vspace{0.2cm}

Let us now illustrate Algorithm 2 with the curve \( \mathcal{C}_f \) from Example \ref{ex-sb}.

\begin{example}\label{ex-cof} Let $f=y^n-x^m+x^{m-k}y^{n-2}\in\mathbb{C}\{x,y\}$ such that $gcd(n,m)=1<n<m$ with $n\neq 2$ and $2\leq k< min\left \{ \frac{(n-2)m}{n},2+\frac{2m}{n}\right \}$. Recall that $\mathcal{C}_f$ admits a parametrization $\varphi(t)=\left (t^n,t^m-\frac{1}{n}t^{(n-1)m-kn}+\sum_{i>(n-1)m-kn}a_it^i\right )$.
	
According to Example \ref{ex-sb}, a Standard basis for $\frac{\Omega^1}{S(f)}$	is $$\mathcal{B}=\{\overline{\eta}_1=\overline{dx}, \overline{\eta_2}=\overline{dy}, \overline{\eta_3}=n\overline{xdy}-m\overline{ydx}\}$$
and computing the minimal $S$-processes for every pair of elements in $\mathcal{B}$ we get
$$S=\{\ s_1(\overline{\eta}_1,\overline{\eta}_2),\ s_2(\overline{\eta}_1,\overline{\eta}_2),\ s_1(\overline{\eta}_1,\overline{\eta}_3),\ s_2(\overline{\eta}_1,\overline{\eta}_3),\ s_1(\overline{\eta}_2,\overline{\eta}_3),\ s_2(\overline{\eta}_2,\overline{\eta}_3)\ \}.$$ 
Since

		$\bullet\ s_1(\overline{\eta}_1,\overline{\eta}_2)=n\overline{x}d\overline{y}-m\overline{y}d\overline{x}=\overline{\eta}_3,$ \vspace{0.3cm}
		
		$\bullet\ s_2(\overline{\eta}_1,\overline{\eta}_2)=n\overline{y}^{n-1}d\overline{y}-m\overline{x}^{m-1}d\overline{x}=-(m-k)\overline{x}^{m-k-1}\overline{y}^{n-2}d\overline{x}-(n-2)\overline{x}^{m-k}\overline{y}^{n-3}d\overline{y},$ \vspace{0.3cm}
		
		$\bullet\ s_1(\overline{\eta}_1,\overline{\eta}_3)=\frac{(n-2)m-kn}{n}\overline{x}^{m-k}d\overline{x}+\overline{y}\overline{\eta}_3=$
		
		\hspace{1.85cm}
		$=\frac{(n-2)((n-2)m-kn)}{nm}\overline{x}^{m-2k+1}\overline{y}^{n-3}d\overline{y}-\frac{(m-k)(n-2)}{nm}\overline{x}^{m-2k}\overline{y}^{n-3}\overline{\eta}_3,$
		\vspace{0.3cm}
		
		$\bullet\
		s_2(\overline{\eta}_1,\overline{\eta}_3)=\frac{(n-2)m-kn}{n}\overline{y}^{n-1}d\overline{x}+\overline{x}^k\overline{\eta}_3=\frac{2}{n}\overline{y}^{n-2}\overline{\eta}_3,$
		\vspace{0.3cm}
				
		$\bullet\
		s_1(\overline{\eta}_2,\overline{\eta}_3)=\frac{(n-2)m-kn}{m}\overline{y}^{n-2}d\overline{y}+\overline{x}^{k-1}\overline{\eta}_3-\frac{(m-k)((n-2)m-kn)}{nm}\overline{x}^{m-k-1}\overline{y}^{n-3}d\overline{x}+$
		
		\hspace{2.5cm}
		$+\frac{(m-k)(n-2)((n-2)m-kn)}{nm^2}\overline{x}^{m-2k}\overline{y}^{2n-6}d\overline{y}-\frac{(n-2)(m-k)^2}{nm^2}\overline{x}^{m-2k-1}\overline{y}^{2n-6}\overline{\eta}_3,$
		\vspace{0.3cm}
		
		$\bullet\
		s_2(\overline{\eta}_2,\overline{\eta}_3)=\frac{(n-2)m-kn}{m}\overline{x}^{m-k+1}d\overline{y}+\overline{y}^2\overline{\eta}_3=-\frac{k}{m}\overline{x}^{m-k}\overline{\eta}_3$.
		\vspace{0.3cm}
		
Thus, a generating set for $S(f)$ is
$$\begin{array}{ll}
G = & \left \{ nxdy-mydx-\eta_3=0,\right .\vspace{0.15cm}\\  & \hspace{0.2cm} \omega_0:=ny^{n-1}dy-mx^{m-1}dx+(m-k)x^{m-k-1}y^{n-2}dx+(n-2)x^{m-k}y^{n-3}dy=df,\vspace{0.15cm}\\
& \hspace{0.2cm} \omega_1:=\frac{(n-2)m-kn}{n}x^{m-k}dx+y\eta_3+ \\
&\hspace{1.5cm}-\frac{(n-2)((n-2)m-kn)}{nm}x^{m-2k+1}y^{n-3}dy+\frac{(m-k)(n-2)}{nm}x^{m-2k}y^{n-3}\eta_3, \vspace{0.15cm}\\	
& \hspace{0.2cm} \omega_3:= \frac{(n-2)m-kn}{n}y^{n-1}dx+x^k\eta_3-\frac{2}{n}y^{n-2}\eta_3 \vspace{0.15cm}\\
& \hspace{0.2cm} \omega_2:= \frac{(n-2)m-kn}{m}y^{n-2}dy+x^{k-1}\eta_3+\frac{(m-k)((n-2)m-kn)}{nm}x^{m-k-1}y^{n-3}dx+ \\
&\hspace{1.5cm}-\frac{(m-k)(n-2)((n-2)m-kn)}{nm^2}x^{m-2k}y^{2n-6}dy+\frac{(n-2)(m-k)^2}{nm^2}x^{m-2k-1}y^{2n-6}\eta_3 \vspace{0.15cm} \\
&\hspace{0.2cm}\left . \omega_4:=\frac{(n-2)m-kn}{m}x^{m-k+1}dy+y^2\eta_3+\frac{k}{m}x^{m-k}\eta_3\right \}.	
\end{array}$$
 
Since $\omega_1\wedge\omega_2=-((n-2)m-kn)fdx\wedge dy$, it follows that $\{\omega_1,\omega_2\}$ is a Saito basis for $S(f)$. In particular, we have
$$\omega_1\wedge df=-nm\left ( y+\frac{(m-k)(n-2)}{nm}x^{m-2k}y^{n-3}\right )fdx\wedge dy\ \ \mbox{and}$$
$$\omega_2\wedge df=-nm\left ( x^{k-1}+\frac{(m-k)^2(n-2)}{nm^2}x^{m-2k-1}y^{2n-6}\right )fdx\wedge dy.$$
\end{example}

According to Berger (see \cite{berger}), for an irreducible $f\in\mathbb{C}\{x,y\}$, we get $$\dim_{\mathbb{C}}\mathcal{T}_f=\mu(f)-\sharp (\Lambda(f)\setminus\Gamma(f)).$$ So, by (\ref{tjurina-torsion}) and Proposition \ref{tjurina}, we get
$$\sharp (\Lambda(f)\setminus\Gamma(f))=I(h_1,h_2),$$
where $\langle h_1,h_2\rangle$ is the cofactor ideal of $S(f)$. 

\begin{remark}\label{tau=3}
For $f=y^n-x^m+x^{m-k}y^{n-2}\in\mathbb{C}\{x,y\}$ with $gcd(n,m)=1<n<m$, $n\neq 2$ and $2\leq k< min\left \{ \frac{(n-2)m}{n},2+\frac{2m}{n}\right \}$, the Example \ref{ex-sb} and Example \ref{ex-cof}  we obtain
$\sharp(\Lambda(f)\setminus\Gamma(f))=k-1=I(h_1,h_2)$ 
and consequently, since $\mu(f)=(m-1)(n-1)$, we obtain $$\tau(f)=(m-1)(n-1)-(k-1).$$

Notice that this case includes all curves $\mathcal{C}_f$ with $\mu(f)-\tau(f)=1$ (see \cite{bayer})
	and all curves with multiplicity equal to three that are not
	analytically equivalent to the monomial curves (see
	\cite{multifour}). 
\end{remark}

Cano, Corral and Senovilla-Sanz (see \cite{felipe-nuria-david}) show how to obtain an element of a Saito basis for curves with semigroup $\langle v_0,v_1\rangle$ using results that can be reinterpreted in terms of syzygies of the elements in a Standard basis for $\frac{\Omega^1}{S(f)}$. Therefore, Algorithm 2 can be considered as a generalization of the method presented in \cite{felipe-nuria-david}. 

Furthermore, in \cite{carvalho}, the first and third authors present an algorithm to compute a Standard basis for $\frac{\Omega^1}{S(f)}$ for any reduced $f\in\mathbb{C}\{x,y\}$. Proceeding as above, we can obtain a similar algorithm to compute a Saito basis for $S(f)$. 

An alternative way to obtain a Saito basis for $S(f)$ for any reduced $f=f_1f_2\in\mathbb{C}\{x,y\}$ is to find a generating set for $S(f_1)\cap S(f_2)$ (see for instance methods to compute module intersections in \cite{singular}). Indeed, we have $S(f_1f_2)=S(f_1)\cap S(f_2)$, as shown in the next lemma.

\begin{lemma}
Let $f=f_1f_2\in\mathbb{C}\{x,y\}$ be reduced. Then $S(f)=S(f_1)\cap S(f_2)$. 
\end{lemma} 
\Dem If $\omega\in S(f)$ then $\omega\wedge df=\omega\wedge (f_1df_2+f_2df_1)=hfdx\wedge dy=hf_1f_2dx\wedge dy$ for some $h\in\mathbb{C}\{x,y\}$, that is 
$$f_i(\omega\wedge df_j)=f_j(hf_idx\wedge dy-\omega\wedge df_i)\ \ \mbox{for}\ \ \{i,j\}=\{1,2\}.$$
Since $gcd(f_1,f_2)=1$, we obtain $\omega\wedge df_i=h_if_idx\wedge dy$, so that $\frac{\omega\wedge df_i}{dx\wedge dy}\in\langle f_i\rangle$ for $i=1,2$. Hence, $\omega\in S(f_1)\cap S(f_2)$.

Conversely, if $\omega\wedge df_i=h_if_idx\wedge dy$ for $i=1,2$, then
$$\omega\wedge df=\omega\wedge (f_1df_2+f_2df_1)=(f_1h_2f_2+f_2h_1f_1)dx\wedge dy=(h_1+h_2)fdx\wedge dy,$$
which implies $\omega\in S(f)$.
\cqd

In some cases, a Saito basis for $S(f_1f_2)$ can be obtained from a Saito basis for $S(f_1)$, as shown in the following theorem.

\begin{theorem}\label{theo-f1f2}
Let \(f = f_1 f_2 \in \mathbb{C}\{x,y\}\), and suppose that \(\{\omega_1,\omega_2\}\) is a Saito basis for \(S(f_1)\).  Assume futhermore that, for each \(i=1,2\),
$ \omega_i\wedge df_2 \;=\;(R_i + H_i\,f_2)\,dx\wedge dy$,
with \(R_i\not\in\langle f_2\rangle\).  If $f_2\in\langle \frac{R_1}{G},\frac{R_2}{G}\rangle$, where 
$G = \gcd(R_1,R_2)$,
then we can choose \(S_1,S_2\in\mathbb{C}\{x,y\}\) such that
$S_1\,\frac{R_1}{G}\;+\;S_2\,\frac{R_2}{G}\;=\;v\,f_2$
for some unit \(v\in\mathbb{C}\{x,y\}\).  Defining
$\eta_1 = \frac{1}{G}\bigl(R_2\,\omega_1 - R_1\,\omega_2\bigr),
\eta_2 = S_1\,\omega_1 + S_2\,\omega_2,
$
we obtain a Saito basis \(\{\eta_1,\eta_2\}\) of \(S(f_1f_2)\).
\end{theorem}
\Dem
Since $\{\omega_1,\omega_2\}$ is a Saito basis for $S(f_1)$, by Saito's criterion there exist $h_i,u\in\mathbb{C}\{x,y\}$ with $u$ a unit, such that $\omega_1\wedge\omega_2=uf_1dx\wedge dy$ and $\omega_i\wedge df_1=h_if_1dx\wedge dy$ for $i=1,2$.

Taking $R_1, R_2, G, S_1, S_2\in\mathbb{C}\{x,y\}$ satisfying the hypothesis and considering
$$\eta_1 = \frac{1}{G}\bigl(R_2\,\omega_1 - R_1\,\omega_2\bigr),\ \ \
\eta_2 = S_1\,\omega_1 + S_2\,\omega_2,
$$ we get
$$\begin{array}{ll}\eta_1\wedge d(f_1f_2) & =f_2(\eta_1\wedge df_1)+f_1(\eta_1\wedge df_2)\\
	& = \left (f_2\left ( \frac{1}{G}R_2h_1f_1-\frac{1}{G}R_1h_2f_1\right )+f_1\left ( \frac{1}{G}R_2(R_1+H_1f_2)-\frac{1}{G}R_1(R_2+H_2f_2)\right )\right )dx\wedge dy \\
	& =\frac{1}{G}\left ( (h_1+H_1)R_2-(h_2+H_2)R_1\right )f_1f_2dx\wedge dy
	\\
	& \\
\eta_2\wedge d(f_1f_2) & = f_2 (\eta_2\wedge df_1)+f_1(\eta_2\wedge df_2) \\
& =f_2\left ( S_1h_1f_1+S_2h_2f_1\right )dx\wedge dy+f_1\left ( S_1(R_1+H_1f_2)+S_2(R_2+H_2f_2)\right )dx\wedge dy \\
& = (S_1(h_1+H_1)+S_2(h_2+H_2))f_1f_2 dx\wedge dy+(S_1R_1+S_2R_2)f_1 dx \wedge dy.
	\end{array}$$
	Since $S_1R_1+S_2R_2=Gvf_2$, it follows that $\eta_1,\eta_2\in S(f_1f_2)$. Finally, we compute:
	$$\begin{array}{ll}
	\eta_1\wedge \eta_2 & = \frac{1}{G}R_2\omega_1\wedge \left (S_1\omega_1+S_2\omega_2\right )-\frac{1}{G}R_1\omega_2\wedge\left (S_1\omega_1+S_2\omega_2\right )\\
	& = \frac{1}{G}R_2S_2uf_1dx\wedge dy+\frac{1}{G}R_1S_1uf_1dx\wedge dy=\left ( \frac{R_2S_2+R_1S_1}{G}\right )uf_1dx\wedge dy=uvf_1f_2dx\wedge dy	,	
		\end{array}$$
		so, by the Saito's criterion, $\{\eta_1,\eta_2\}$ is a Saito basis for $S(f_1f_2)$.
\cqd

\begin{example}
	Let $f_1=y^n-x^m, f_2=x^n-y^m\in\mathbb{C}\{x,y\}$ with $n<m$.
	
	Since $f_1$ is quasi-homogeneous, by Example \ref{qh}, the set $\{\omega_1=mydx-nxdy, \omega_2=df_1=ny^{n-1}dy-mx^{m-1}dx\}$ is a Saito basis for $S(f_1)$.
	
	We have that $$\omega_1\wedge df_2=((n^2-m^2)y^m+n^2f_2)dx\wedge dy\ \ \ \omega_2\wedge df_2=x^{n-1}y^{n-1}(-n^2+m^2x^{m-n}y^{m-n})dx\wedge dy.$$
With the notation of Theorem \ref{theo-f1f2}, we otain $R_1=(n^2-m^2)y^m$, $R_2=x^{n-1}y^{n-1}(-n^2+m^2x^{m-n}y^{m-n})$ and $G=gcd(R_1,R_2)=y^{n-1}$.
	
	Taking $S_1=y^{n-1}(-n^2+m^2x^{m-n}y^{m-n})$ and $S_2=(m^2-n^2)x$ we obtain $$S_1\frac{R_1}{G}+S_2\frac{R_2}{G}=(m^2-n^2)(-n^2+m^2x^{m-n}y^{m-n})f_2.$$
	Thus, $$\left \{\eta_1=x^{n-1}(-n^2+m^2x^{m-n}y^{m-n})\omega_1+(m^2-n^2)y^{m-n+1}\omega_2,\right .$$
$$\left .\eta_2=y^{n-1}(-n^2+m^2x^{m-n}y^{m-n})\omega_1+(m^2-n^2)x\omega_2\right \}$$ is a Saito basis for $S(f_1f_2)$. In fact, we have
$$\eta_1\wedge d(f_1f_2)=nx^{n-1}(n+m)(-n^2+m^2x^{m-n}y^{m-n})f_1f_2dx\wedge dy$$
$$\eta_2\wedge d(f_1f_2)=(m^2-n^2)y^{n-1}(-n^2+m^2x^{m-n}y^{m-n})f_1f_2dx\wedge dy$$
with $\eta_1\wedge \eta_2=nm(m^2-n^2)(-n^2+m^2x^{m-n}y^{m-n})f_1f_2dx\wedge dy$.
\end{example} 

A particular case of Theorem \ref{theo-f1f2} is given in the following corollary.

\begin{corollary}\label{yf}
Let $\{\omega_1=A_1dx+B_1dy,\omega_2=A_2dx+B_2dy\}$ be a Saito basis for $S(f)$. Suppose such that $\alpha=mult(A_1(x,0))\leq mult(A_2(x,0))$, then $\left \{\eta_1:=\frac{A_2(x,0)}{x^{\alpha}}\omega_1-\frac{A_1(x,0)}{x^{\alpha}}\omega_2,\eta_2:=y\omega_1\right \}$ is a Saito basis for $S(yf)$. Moreover, writing $\omega_i\wedge df=h_ifdx\wedge dy$ and $A_i=A_i(x,0)+yH_i$ we get
$$\eta_1\wedge d(yf)=\left ( \frac{A_2(x,0)(h_1+H_1)-A_1(x,0)(h_2+H_2)}{x^{\alpha}}\right )yfdx\wedge dy$$
$$\mbox{and}\ \ \eta_2\wedge d(yf)=(yh_1+A_1)yfdx\wedge dy.$$
\end{corollary}
\Dem Taking $f_2=y$ in Theorem \ref{theo-f1f2} we have:
$$\omega_i\wedge df_2=\omega_i\wedge dy=A_idx\wedge dy=(A_i(x,0)+yH_i)dx\wedge dy\ \ \ \mbox{for}\ i=1,2.$$

Without loss of generality, assume that $\alpha=mult(A_1(x,0))\leq mult(A_2(x,0))$, that is $x^{\alpha}=gcd(A_1(x,0),A_2(x,0))$.

Since $\frac{A_1(x,0)}{x^{\alpha}}\in\mathbb{C}\{x\}$ is a unit and $$\frac{A_1(x,0)}{x^{\alpha}} f_2=y\cdot \frac{A_1(x,0)}{x^{\alpha}}+0\cdot \frac{A_2(x,0)}{x^{\alpha}}.$$
So, taking $R_1=A_1(x,0), R_2=A_2(x,0), G=x^{\alpha}, S_1=y$ and $S_2=0$ in Theorem \ref{theo-f1f2} we obtain
$$\eta_1=\frac{A_2(x,0)}{x^{\alpha}}\omega_1-\frac{A_1(x,0)}{x^{\alpha}}\omega_2\ \ \ \eta_2=y\omega_1+0\omega_2=y\omega_1$$
is a Saito basis for $S(yf)$ with
$$\eta_1\wedge d(yf)=\left ( \frac{A_2(x,0)(h_1+H_1)-A_1(x,0)(h_2+H_2)}{x^{\alpha}}\right )yfdx\wedge dy$$
$$\eta_2\wedge d(yf)=(yh_1+A_1)yfdx\wedge dy$$
and
$\eta_1\wedge \eta_2=\frac{A_1(x,0)}{x^{\alpha}}uyfdx\wedge dy$ where $\omega_1\wedge\omega_2=ufdx\wedge dy$ with $u\in\mathbb{C}\{x,y\}$ a unit.
\cqd

\section{Saito basis and Tjurina number for curves $\mathcal{C}_f$ with $m(f)\leq 3$}

Combining the results of previous section, we are able to present a Saito basis $\{\omega_1,\omega_2\}$ for plane curves $\mathcal{C}_f$ with multiplicity $m(f)\leq 3$. According to Proposition \ref{tjurina}, we have $\tau(f)=\mu(f)-I(h_1,h_2)$, where $\omega_i\wedge df=h_ifdx\wedge dy$ for $i=1,2$. Thus, as a biproduct, we can compute the Tjurina number for any plane curve with $m(f)\leq 3$.

Recall that if $f=\prod_{i=1}^{r}f_i$ then $\mu(f)=\sum_{i=1}^{r}\mu(f_i)+2\sum_{1\leq i<j\leq r}I(f_i,f_j)-r+1$.

\subsection{Multiplicity one}

If $f=\alpha x+\beta y+h\in\mathbb{C}\{x,y\}$ with $h\in\langle x,y\rangle^2$, then $\tau(f)=0$ and Remark \ref{equi-zero} provides that $\{fdg,df\}$ is a Saito basis for $S(f)$ for any 
$g=\gamma x+\delta y$ such that $\alpha\delta-\beta\gamma\neq 0$.

\subsection{Multiplicity two}

If $f\in\mathbb{C}\{x,y\}$ has $m(f)=2$, then $\tau(f)=\mu(f)$ and we have the following possibilities:

\begin{itemize}
	\item $f$ is irreducible.
	
	In this case, up to a change of coordinates, we may suposse that $f=y^2-x^n$ where $2<n\in\mathbb{N}$ is odd. In this case, $\tau(f)=\mu(f)=n-1$ and, by Example \ref{qh}, $\{2xdy-nydx,df\}$ is a Saito basis for $S(f)$.
	
	\item $f$ has two transversal components.
	
	Up to a change of coordinates, we may consider $f=xy$. So, $\tau(f)=\mu(f)=2I(x,y)-2+1=1$ and, by Example \ref{qh}, a Saito basis for $S(f)$ is
	$\{xdy-ydx,df\}.$
	
	\item $f$ has two components with same tangent.
	
	In this case, after a suitable change of coordinates, we may assume that $f=y(y-x^n)$ with $I(y,y-x^n)=n>1$, hence $\tau(f)=\mu(f)=2I(y,y-x^n)-2+1=2n-1$. Since $f$ is quasi-homogeneous, by Example \ref{qh},  
	$\{xdy-nydx,df\}$ is a Saito basis for $S(f)$.  
\end{itemize}

\subsection{Multiplicity three}

In this case, we have $f=\prod_{i=1}^{r}f_i$ with each $f_i$ irreducible and $1\leq r\leq 3$.

\begin{itemize}
	\item $r=1$.
	
	In this case, $f$ is irreducible. According to \cite{bayer} and \cite{multifour}, up to a change of coordinates we have two possibilities:
	\begin{itemize}
		\item[1)] $f=y^3-x^n$ with $1=gcd(3,n)<3<n$.
		
		Since $f$ is quasi-homogeneous, we have $\tau(f)=\mu(f)=2(n-1)$, and by Example \ref{qh}, a Saito basis for $S(f)$ is: $\{3xdy-nydx,df\}$.
		
		\item[2)] $f=y^3-x^m+x^{m-k}y$ with $2\leq k\leq\left [\frac{m}{3}\right ]$.
		
	This case is a particular instance of Example \ref{ex-cof}. A Saito basis is given by
		$$\begin{array}{c}
		\{\omega_1=\left ( \frac{m-3k}{3}x^{m-k}-my^2-\frac{m-k}{3}x^{m-2k}y\right )dx+\left ( 3xy+\frac{2}{3}x^{m-2k+1}\right )dy, \vspace{0.25cm}\\ \omega_2=\left (-mx^{k-1}y+\frac{(m-k)(m-3k)}{3m}x^{m-k-1}\right )dx+\left (\frac{m-3k}{m}y+3x^k\right )dy\}\end{array}$$
		and, according to Remark \ref{tau=3}, we have $\tau(f)=2(m-1)-(k-1)=2m-k-1$.
	\end{itemize}
	
	\item $r=2$.
	
	In this case, we may assume that $f=yg$ where $m(g)=2$ and the semigroup of $\mathcal{C}_g$ is $\Gamma(g)=\langle 2,n\rangle$ with $1=gcd(2,n)<2<n$. We have the following possibilities:
	
	\begin{itemize}
		\item[1)] $\mathcal{C}_g$ has tangent $x=0$.
		
		In this situation, up to a change of coordinates, we may suppose that $f=y(x^2-y^n)$. Notice that $f$ is quasi-homogeneous and $\tau(f)=\mu(f)=\mu(y)+\mu(x^2-y^n)+2I(y,x^2-y^n)-r+1=n+2$. By Example \ref{qh}, a Saito basis for $S(f)$ is $\{nxdy-2ydx,df\}.$
		
		\item[2)] $\mathcal{C}_g$ has tangent $y=0$.
		
		Now we consider different cases according to $I(y,g)$:
		
		\begin{itemize}
			\item[a)] $I(y,g)=2m<n$.
			
			In this case, $\mathcal{C}_g$ admits a parametrization (up to a change of coordinates) $(t^2,t^{2m}+t^n)$, that is, $g=(y-x^m)^2-x^n=y^2-2x^my+x^{2m}-x^n$. By Example \ref{m-2}, a Saito basis for $S(g)$ is
			$\{\omega_1=2xdy-(ny+(2m-n)x^m)dx,\omega_2=dg=2(y-x^m)dy-(2mx^{m-1}(y-x^m)+nx^{n-1})dx\}$ with $\omega_1\wedge\omega_2=-2ngdx\wedge dy$ and, according to Corollary \ref{yf}, a Saito basis for $S(f)=S(yg)$ is
			$$\begin{array}{l}
			\left \{ \eta_1=(2mx^{m-1}-nx^{n-m-1})\omega_1-(n-2m)\omega_2\right .\\
			\hspace{0.75cm}=y(n^2x^{n-m-1}-4m^2x^{m-1})dx+2((2m-n)y+nx^m-nx^{n-m})dy\vspace{0.2cm}
			\\
			\left . \hspace{0.2cm}\eta_2=y\omega_1=2xydy-y(ny+(2m-n)x^m)dx \right \}
			\end{array}$$
			with
			$$\begin{array}{l}
			\eta_1\wedge d(yg) =(-4m(n-m)x^{m-1}+3n^2x^{n-m-1})ygdx\wedge dy=k_1ygdx\wedge dy\\ \eta_2\wedge d(yg) = (-3ny-2mx^m+nx^n)ygdx\wedge dy=k_2ygdx\wedge dy
		\end{array}$$
		In particular, $$\tau(f)=\mu(f)-I(k_1,k_2)=0+(n-1)+2I(y,g)-2+1-I(k_1,k_2)=n+3m-1.$$
			
			\item[b)] $I(y,g)=n$.
			
			In this case, up to a change of coordinates, $\mathcal{C}_g$ admits parametrization $(t^2,t^n)$ or $(t^2,t^n+t^{n+2m-1})$ for some $1\leq m\leq \frac{n-3}{2}$ that is, $g=y^2-x^n$ or $g=y^2-2x^{\frac{n+2m-1}{2}}y-x^n+x^{n+2m-1}$.
			
			\begin{itemize}
				\item[i)] $g=y^2-x^n$.
				
				Since $f=yg=y(y^2-x^n)$ is quasi-homogeneous, by Example \ref{qh}, a Saito basis for $S(f)$ is $\{nydx-2xdy,df\}$ and $\tau(f)=\mu(f)=n-1+2I(y,g)-2+1=3n-2$.
				
				\item[ii)] $g=y^2-2x^{\frac{n+2m-1}{2}}y-x^n+x^{n+2m-1}$.
				
				According to Example \ref{m-2}, $\{\omega_1=2xdy-(ny+(2m-1)x^{\frac{n+2m-1}{2}})dx,\omega_2=dg=2(y-x^{\frac{n+2m-1}{2}})dy-((n+2m-1)x^{\frac{n+2m-3}{2}}(y-x^{\frac{n+2m-1}{2}})+nx^{n-1})dx\}$ is a Saito basis for $S(g)$ with $\omega_1\wedge\omega_2=-2ngdx\wedge dy$ and consequently, by Corollary \ref{yf}, a Saito basis for $S(f)=S(yg)$ is given by

				$$\begin{array}{l}
					\left \{ \eta_1=((n+2m-1)x^{\frac{n+2m-3}{2}}-nx^{\frac{n-2m-1}{2}})\omega_1+	(2m-1)\omega_2\right .\\
					\hspace{0.75cm}=2((2m-1)y-nx^{\frac{n-2m+1}{2}}+nx^{\frac{n+2m-1}{2}})dy+\\  \hspace{2.5cm}+(n^2x^{\frac{n-2m-1}{2}}-(n+2m-1)^2x^{\frac{n+2m-3}{2}})ydx\vspace{0.2cm}
					\\
					\left . \hspace{0.2cm}\eta_2=y\omega_1=2xydy-y(ny+(2m-1)x^\frac{n+2m-1}{2})dx \right \}
				\end{array}$$
				with
				$$\begin{array}{l}
					\eta_1\wedge d(yg) =(3n^3x^{\frac{n-2m-1}{2}}-2(m+1)(n+2m-1)x^{\frac{n+2m-3}{2}})ygdx\wedge dy=k_1ygdx\wedge dy\\ \eta_2\wedge d(yg) = (-3ny-(2m-1)x^{\frac{n+2m-1}{2}})ygdx\wedge dy=k_2ygdx\wedge dy
				\end{array}$$
				Thus, we obtain $$\tau(f)=\mu(f)-I(k_1,k_2)=0+(n-1)+2I(y,g)-2+1-I(k_1,k_2)=\frac{5n+2m-3}{2}.$$				
			\end{itemize}
		\end{itemize}
	\end{itemize}
	
	\item $r=3$.
	
	In this case, $f=f_1f_2f_3$ with $m(f_i)=1$ for $1\leq i\leq 3$. Notice that $\mu(f)=2(I(f_1,f_2)+I(f_1,f_3)+I(f_2,f_3))-2$ and we have the following possibilities:
	\begin{itemize}
		\item[1)] $f$ has at least two distinct tangent lines.
		
		Without loss of generality, we may assume that $I(f_1,f_2)=1$ and $I(f_1,f_3)=n\geq 1$. In particular, we get $I(f_2,f_3)=1$ and, up to a change of coordinates, we can consider $f=xy(y-x^n)$ that is quasi-homogeneous. In this way, $\tau(f)=\mu(f)=2n+2$ and a Saito basis for $S(f)$ is, according to Example \ref{qh}, given by $\{nydx-xdy,df\}$.

		\item[2)] $f$ has just one tangent line.
		
		In this case, we suppose that $1<I(f_1,f_2)=n\leq m=I(f_1,f_3)$. In particular, we have $I(f_2,f_3)\geq n$ and we have the following possibilities.
		
		\begin{enumerate}
			\item[a)] $I(f_1,f_2)=n<m= I(f_1,f_3)$.
			
			In this situation, we get $I(f_2,f_3)=n$ and up to a change of coordinates, we may assume that $f=y(y-x^n)(y-x^m)$. Let us denote $g=(y-x^n)(y-x^m)=y^2-(x^n+x^m)y+x^{n+m}$. By Example \ref{m-2} we can determine a Saito basis for $S(g)$ and using it, by Corollary \ref{yf} we obtain a Saito basis for $S(f)=S(yg)$ in the form
			$$\begin{array}{l}
				\left \{ \eta_1=4\left ((m-n)y+nx^n+h.o.t\right )dy-\left (4nmx^{n-1}y+h.o.t.\right )dx,\right.\\
				\hspace{0.2cm}\left .\eta_2=2xydy-(2ny+
                2(m-n)x^m+h.o.t.)ydx\right \}\end{array}$$ such that
			$$\eta_1\wedge df=\left ( -4n(2m+n)x^{n-1}+h.o.t.\right )fdx\wedge dy=k_1fdx\wedge dy\ \ \mbox{and}$$
			$$\eta_2\wedge df=\left (-6ny-(m-n)x^m+h.o.t.\right )fdx\wedge dy=k_2fdx\wedge dy.$$ In this way, we get $\tau(f)=\mu(f)-I(k_1,k_2)=4n+2m-2-(n-1)=3n+2m-1.$
									
			\item[b)] $I(f_1,f_2)=n= I(f_1,f_3)$.
			
			Notice that $I(f_2,f_3)=n$. Up to change of coordinates, we can consider $f=y(y-x^n)(y+x^n-ax^k)$ with $n<k<2n-1$ and we have the possibilities:
			\begin{enumerate}
				\item[i)] $a=0$.
				
				In this case $f$ is quasi-homogeneous and $\tau(f)=\mu(f)=6n-2$. By Example \ref{qh}, a Saito basis for $S(f)$ is $\{nydx-xdy,df\}$. 
				
				\item[ii)] $a\neq 0$.
				
				Denoting $g=y^2-ax^ky-(x^{2n}-ax^{k+n})$, we compute a Saito basis for $S(g)$ and, by Corollary \ref{yf}, we can obtain a Saito basis for $S(f)=S(yg)$ given by
					$$\begin{array}{l}
					\left \{ \eta_1=\left (2a(k-n)y-4nx^{2n-k}-a^2(k-n)x^k+h.o.t\right )dy+\right .\\
					\hspace{1.5cm}+\left ((4n^2x^{2n-k-1}-ka^2(k-n)x^{k-1})y+h.o.t.\right )dx,\\
					\hspace{0.2cm}\left .\eta_2=2xydy-(2ny+a(k-n)x^k+h.o.t.)ydx\right \}\end{array}$$ such that
				$$\eta_1\wedge df=\left ( 12n^2x^{2n-k-1}+h.o.t.\right )dx\wedge dy=k_1fdx\wedge dy\ \ \mbox{and}$$
				$$\eta_2\wedge df=\left (-6ny+h.o.t.\right )dx\wedge dy=k_2fdx\wedge dy.$$
				Hence, $\tau(f)=\mu(f)-I(k_1,k_2)=4n+k-1$.
				
			\end{enumerate}
			
		\end{enumerate}
	\end{itemize}

\end{itemize}

\vspace{1cm}

\begin{center}
\begin{tabular}{lclcl}
	Carvalho, E. & & Fern\'{a}ndez-S\'{a}nchez, P. & & Hernandes, M. E. \\
	{\it emilio.carvalho$@$ufmt.br} & & {\it pefernan$@$pucp.edu.pe} & & {\it mehernandes$@$uem.br} 
	\\
	INMA-UFMT & & PUCP & & DMA-UEM \\
	Av. Fernando Corr\^ea da Costa 2367 & & Av. Universit\'{a}ria 1801 &  & Av. Colombo 5790 \\
	Cuiab\'a-MT 78060-900 & & San Miguel, Lima 32 & & Maring\'{a}-PR 87020-900 \\
	Brazil & & Peru &  & Brazil
\end{tabular}
\end{center}


\begin{thebibliography}{Referen}
	
	\bibitem{Almiron} {\sc Almir\'on, P. and Hernandes, M. E.}, {\it The Tjurina number of a plane curve with two branches and high intersection multiplicity}, arXiv:2501.12836.
	
	\bibitem{BGHH} {\sc Bayer, V. A. S.; Guzm\'am, E. M. N.; Hefez, A. and Hernandes, M. E.} {\it Tjurina number of a local complete intersection curve}. Comm. Algebra 53(2), 509-520 (2025).
	
	\bibitem{brunela} {\sc Brunella, M.}, {\it Some remarks on indices of holomorphic vector fields}, Publ. Mat. 41, 527–544 (1997).
	
	\bibitem{bayer} {\sc Bayer, V. and Hefez, A.}, {\it Algebroid plane curves whose Milnor and Tjurina numbers differ by one or two}, Bul. Soc. Bras. Mat. 32, 1, 63-81. (2001).
	
	\bibitem{berger} {\sc Berger, R.}, {\it Differentialmoduln eindimensionaler lokaler Ringe}. Math. Z. 81, 326-354 (1963).
	
	\bibitem{felipe-nuria-david} {\sc Cano, F.; Corral, N. and Senovilla-Sanz, D.}, {\it Computing a Saito basis from a standard basis}. arXiv:2404.00316.
	
	\bibitem{carvalho1} {\sc Carvalho, E. and Hernandes, M. E.}; {\it The value semiring of an algebroid curve}, Comm. Algebra	48, 3275–3284 (2020).
	
	\bibitem{carvalho} {\sc Carvalho, E. and Hernandes, M. E.}, {\it Standard bases for fractional ideals of the local ring of an algebroid curve}, Journal of Algebra 551, 342-361 (2020).
	
	\bibitem{dimca} {\sc Dimca, A. and Greuel, G.-M.}, {\it On 1-forms on isolated complete intersection on curve singularities}, J. Singularities 18, 114-118 (2018).
	
	\bibitem{yohann} {\sc Genzmer, Y.}, {\it The Saito module and the moduli of a germ of curve in $(\mathbb{C}^2,0)$}. To appear in Annales de L'Institut Fourier.
    
    \bibitem{yohann2} {\sc Genzmer, Y.}, {\it Number of moduli for a union of smooth curves in $\mathbb{C}^2$}. J. Symb. Comput. 113, 148-170 (2022).
	
	\bibitem{yohann-marcelo} {\sc Genzmer, Y. and Hernandes, M. E.}, {\it On the Saito number of plane curves}, arXiv:2406.13729.
	
	\bibitem{transaction} {\sc Genzmer, Y. and Hernandes, M. E.}; {\it On the Saito basis and the Tjurina number for plane branches}, Trans. Amer. Math. Soc. 373, 3693-3707, (2020).

	\bibitem{singular} {\sc Greuel, G.-M., Pfister, G.}, {\it A Singular Introduction to Commutative Algebra (with contributions by O. Bachmann, C. Lossen, and H. Sch\"onemann)}. Second edition, Springer-Verlag (2008).
	
	\bibitem{basestandard} {\sc Hefez, A. and Hernandes,
		M. E.}, {\it Standard bases for local rings of branches and their module of differentials}, J. Symb. Comp. 42, 178-191 (2007).
	
	\bibitem{comprimento} {\sc Hefez, A. and Hernandes, M. E.}, {\it Colengths of fractional ideals and Tjurina number of a reducible plane curve} arXiv:2409.11153.
	
	\bibitem{HR} {\sc Hernandes, M. E. and Rodrigues Hernandes, M. E.}, {\it The analytic classification of plane curves}, Compositio Mathematica 160(4), 915-944 (2024).
	
	\bibitem{multifour} {\sc Hefez, A. and Hernandes, M. E.}, {\it Analytic classification of plane branches up to multiplicity $4$}, J. Symb. Comp. 44, 626-634 (2009).
	
	
	\bibitem{Delgado} {\sc Mata, F. D.}, {\it The semigroup of values of a curve singularity with several branches}, Manuscripta
	Math. 59, 347-374 (1987).
	
	\bibitem{michler} {\sc Michler, R. I.}, {\it Torsion of differentials of hypersurfaces with isolated singularities}, J. Pure and Applied Algebra 104, 81-88 (1995).
	
	\bibitem{tajima1} {\sc Nabeshima, K. and Tajima, S.}, {\it Computation methods of logarithmic vector fields associated to semi-weitghted homogeneous isolated hypersurface singularities}. Tsuku J. Math. 42(2), 191-231 (2018).
	
	\bibitem{tajima2} {\sc Nabeshima, K. and Tajima, S.}, {\it A new algorithm for computing logarithmic vector fields along an isolated singularity and Bruce-Roberts Minor ideals}. Journal of Symbolic Computation 107, 190-208 (2021).
	
	\bibitem{Pol} {\sc Pol, D.}, {\it On the values of logarithmic residues along curves}, Annales de L'Institut Fourier, Tome 68(2), 725-766 (2018).
	
	\bibitem{saito1} {\sc Saito, K.}, {\it Quasihomogene isolierte Singularit\"{a}ten von Hyperfl\"{a}chen}, Inv. Math. 14, 123-142, (1971).
	
	\bibitem{saito} {\sc Saito, K.}, {\it Theory of logarithmic differential forms and
		logarithmic vector fields}, Journal of the Faculty of Science, the University of Tokyo. Sect. 1 A, Mathematics 27 (2), 265 - 291, (1980).
		
	\bibitem{walcher} {\sc Walcher, S.}, {\it Plane polynomial vector fields with prescribed invariant curves}, Proceedings of the Royal Society of Edinburgh (130A), 633 - 649, (2000).
	
	\bibitem{Waldi} {\sc Waldi, R.}, {\it Wertehalbgruppe und Singularit\"aten einer ebenen algebraischen Kurve}, Dissertation, Regensburg (1972).
	
	\bibitem{Z-top} {\sc Zariski, O.}, {\it General theory of saturation and of saturated local rings II: Saturated local rings of dimension 1}, Amer. J. Math. 93, 872-964 (1971).
	
	\bibitem{zariski-book} {\sc Zariski, O.}, {\it The Moduli Problem for Plane Branches} with an appendix by Bernard Teissier. Translated by Ben Lichtin. University Lecture Notes Series 39, AMS (2006).
	
\end{thebibliography}
\end{document}